%% file: main.tex
\documentclass[]{aero}
\usepackage{textcomp}
\usepackage[acronym]{glossaries}
\usepackage{algorithm}
\usepackage{algpseudocode}

\usepackage{amsthm}
\usepackage{amsmath}
\usepackage[table]{xcolor}
\usepackage{booktabs}
\usepackage{makecell}
\newtheorem{remark}{Remark}
\newtheorem{proposition}{Proposition}
\input{glossary.tex}
\hypersetup{colorlinks=true, linkcolor=black, citecolor=black, urlcolor=black}

\usepackage{graphicx}
\usepackage{longtable}
\usepackage{amsmath}
\usepackage[version=4]{mhchem}
\usepackage{siunitx}
\usepackage{longtable,tabularx}
\newacronym{kkt}{KKT}{Karush-Kuhn-Tucker}
\newacronym{licq}{LICQ}{linear independence constraint qualification}
\newacronym{gtoc}{GTOC}{Global Trajectory Optimisation Competition}

\ADsetup{
title    = {Adaptive Node Distribution via Probability Density Functions for Birkhoff Pseudospectral Trajectory Optimisation},
author   = {Andrian A. Yasintus \cor{}, and Xiaofeng Wu},
address  = {
  The University of Sydney, Sydney, NSW 2006, Australia
},
email    = {andrian.yasintus@sydney.edu.au},
abstract = {
  Pseudospectral methods are a powerful framework for spacecraft trajectory optimisation that rely on fixed, problem-independent collocation nodes such as Legendre-Gauss-Lobatto (LGL) nodes. These nodes are located according to polynomial root distributions, which can result in insufficient resolution in regions of high dynamical activity, such as those encountered during collision avoidance manoeuvres. This paper presents an adaptive node distribution strategy for Birkhoff pseudospectral trajectory optimisation, where the collocation nodes are placed according to a probability density function derived from the thrust magnitude profile of the spacecraft. By concentrating nodes in regions of high dynamical variation, the proposed method simultaneously achieves feasibility preservation and improved convergence with fewer nodes in a collision avoidance case study when compared to the universal Birkhoff method, with Model Predictive Control serving as an independent validation comparator. The case study highlights the practical consequences of inter-node constraint violations and validates the adaptive node scheme as a robust and efficient alternative for fuel-optimal spacecraft manoeuvre design.
  },
keywords = {
  Trajectory Optimisation \sep Pseudospectral Methods \sep Optimal Control \sep Spacecraft Guidance \\
},
}

\begin{document}

\maketitle  

\newpage

\section{Introduction}
Spacecraft collision avoidance manoeuvres are one of the most critical operational challenges in trajectory optimisation. The growing population of orbital debris and active satellites in low Earth orbit has significantly increased the frequency of conjunctions, posing a growing threat to operational spacecraft. Therefore, the fuel cost and mission disruption associated with these manoeuvres motivates the development of efficient, reliable, and robust trajectory optimisation methods capable of producing feasible solutions under tight operational and computational constraints. A broad range of methods have been developed to address this challenge, spanning impulsive and continuous-thrust formulations, deterministic and chance-constrained approaches, and single- and multi-encounter scenarios \cite{cammary}. 

Collision avoidance manoeuvres, and spacecraft path planning more broadly, are naturally posed as continuous-time optimal control problems, with the objective of minimising fuel consumption while satisfying dynamic constraints, control bounds, and a spatial keep-out constraint centered on the debris object. Two broad classes of solution method have been employed: indirect methods, which rely on satisfying equations derived from Pontryagin's Minimum Principle such as primer vector theory \cite{lawden}, or direct methods, which transcribe the continuous optimal control problem into a finite-dimensional nonlinear programming (NLP) problem to be solved by a numerical optimiser \cite{betts}. Pseudospectral methods are a distinguished class of direct transcription that approximates states and controls using global orthogonal polynomials evaluated at predetermined quadrature nodes \cite{garg}. Unlike conventional direct methods, pseudospectral methods maintain a connection to necessary conditions through the Covector Mapping Principle (CMP) \cite{rossfahroo_cmp}, providing theoretical grounding for the optimality of the transcribed problem and enabling the recovery of states, controls, and costate variables with spectral accuracy \cite{rossfahroo_costate}, allowing independent verification of optimality through theories like primer vector theory.

Pseudospectral methods have also been continuously improved for higher accuracy and lower computational effort. For example, $hp$-adaptive methods  refine the solution by iteratively partitioning the time domain into subintervals, and adaptively increasing either the number of subintervals ($h$-refinement) or the polynomial degree within each subinterval ($p$-refinement) based on local error estimates, as implemented in solvers such as GPOPS-II \cite{darby, gpops2}. This error-estimation approach remains an active area of development, with recent work refining the mesh adaptation using explicit simulation based error metrics \cite{haman}. However, while effective, the $hp$ approach is inherently reactive as nodes are redistributed in response to observed solution error rather than informed by prior knowledge of the solution structure.

The Birkhoff pseudospectral method is a distinct formulation grounded in Birkhoff interpolation theory \cite{lorentz} which permits interpolation over nodes where only non-consecutive derivative information is available. When applied to Gauss-type quadrature nodes, this gives a set of well-conditioned polynomial basis functions from which the pseudospectral integration matrix (PSIM) can be derived \cite{wang}. The PSIM is the discrete inverse of the pseudospectral differentiation matrix (PSDM), whose ill-conditioning is a well-known limitation of traditional Lagrangian pseudospectral methods \cite{xiao}. This foundation was developed into a complete pseudospectral framework for optimal control, with full compatibility with the CMP while also showing that the Birkhoff formulation supports arbitrary node distributions while retaining the theoretical properties of classical pseudospectral methods including spectral accuracy and costate recovery \cite{ross_birkhoff_part1, ross_birkhoff_part2}. The framework has since found many applications, such as two-stage ascent optimisation \cite{ascent_traj}.

Despite these favourable properties, to the authors' knowledge, the Birkhoff pseudospectral method has not previously been applied with problem-informed adaptive node placement. In collision avoidance problems, dynamic activity is concentrated near the avoidance manoeuvre interval, yet standard node distributions are unable to adapt for this structure. This paper proposes an adaptive Birkhoff pseudospectral method in which collocation nodes are placed according to a probability density function (PDF) derived from the problem structure, with concentrated resolution in regions of high dynamic activity and requires no iterative mesh refinement or error estimation. By exploiting Birkhoff support for arbitrary nodes, the proposed method aims to produce feasible collision avoidance trajectories that satisfy constraints at all points throughout the manoeuvre (not merely at collocation points), addressing a limitation that remains an active concern even in recent constraint-handling work for direct collocation methods \cite{li}.

The remainder of the paper is organised as follows. Section 2 presents the methodology, including the universal Birkhoff pseudospectral formulation, the proposed adaptive node placement transformation, and the resulting changes to the dynamics constraints and cost functional under the new mapping. Section 3 presents the results and discussion, including the collision avoidance simulation outcomes, the limitations of refining the collocation scheme, and comparing the adaptive method against a Model Predictive Control (MPC) baseline. Finally, the conclusions are presented in Section 4.

\section{Methodology}
\subsection{Birkhoff Pseudospectral Formulation}
Consider the fuel-efficient, continuous-thrust flyby optimal control problem:
\begin{equation}\label{eq:opt_eqn}
\begin{aligned}
& \underset{\mathbf{u}}{\text{min}}
& J &= \int_{t_0}^{t_f} ||\mathbf{u}|| dt \\
& \text{subject to}
& \dot{\mathbf{x}} &= f(\mathbf{x,u}) =\begin{bmatrix}
    \mathbf{v} \\
    -\frac{\mu}{||\mathbf{r}||^3}\mathbf{r}+\mathbf{u}
\end{bmatrix}\\
& & \mathbf{r}(t_f) &= \mathbf{r}_\text{target}(t_f)\\
& & ||\mathbf{u}|| &\leq u_\text{max}
\end{aligned}
\end{equation}
where $\mathbf{x} =\mathbf{x}(t) =  [\mathbf{r}(t), \mathbf{v}(t)]^T$ is the state vector (consisting of position $\mathbf{r}$ and velocity $\mathbf{v}$) of the spacecraft, $\mathbf{u} = \mathbf{u}(t)$ is the control thrust vector of the spacecraft, $\mu$ is the standard gravitational parameter, $\mathbf{r}_\text{target}(t_f)$ is the position vector of the target at the rendezvous time $t_f$, and $u_\text{max}$ is the maximum allowable thrust magnitude. 

To solve this problem, we adopt the Birkhoff pseudospectral method as the trajectory transcription scheme, using the universal Birkhoff theory of Ross, Proulx and Borges \cite{birkhoffuniversal} which has been used in implementations for fast trajectory optimisation \cite{ross_birkhoff_part1, ross_birkhoff_part2}. In their method, the trajectory and control are discretized using a collocation method such as Legendre-Gauss-Lobatto (LGL), denoted by $\tau$. We use the LGL collocation scheme here as among the Legendre pseudospectral node distributions, LGL nodes are the most appropriate for finite-horizon optimal control problems with arbitrary boundary conditions such as the problem at hand \cite{rossfahroo_lgl}. The resulting grid consists of $N+1$ nodes formed by the roots of $(1-x^2)P_{N}'(x)$, where $P_N'(x)$ is the derivative of the $N$th Legendre polynomial. Each node has a state $X_k=[R_k, V_k]^T$ and control $U_k$ that can be used for interpolation later to recover the full trajectory and control history. To relate the LGL nodes $\tau$ to physical time $t$, an affine transformation is used:
\begin{align}\label{eq:affine}
    t_k = \frac{t_f-t_0}{2}\tau_k + \frac{t_f+t_0}{2}
\end{align}
Note that in this implementation, the target rendezvous time $t_f$ is left free for the optimiser to determine. This affine transformation therefore acts as a simple scaling relationship between $[-1, 1]$ and $[0, t_f]$ that preserves the original proportion of node spacing.

Now, the general Birkhoff formulation dictates the evolution of the states according to \cite{birkhoffuniversal}:
\begin{align}
    \mathbf{X}&=\mathbf{X_0 1}^T +\mathbf{BV} \\
    \mathbf{V} &=f(\mathbf{X}(\tau), \mathbf{U}(\tau))
\end{align}
where $\mathbf{B}$ is the Birkhoff matrix (the matrix of values of the Birkhoff basis functions evaluated at every node of the grid) and $\mathbf{V}$ is an $(N+1)\times6$ matrix of the dynamics at each node in collocation space. However, because the dynamics in the problem are expressed with respect to physical time $t$ rather than the collocation variable $\tau$, this introduces a scaling factor of $\frac{t_f-t_0}{2}$ to the general Birkhoff formulation due to the chain rule:
\begin{align}
    \frac{\mathrm{d}\mathbf{x}}{\mathrm{d}\tau} = \frac{\mathrm{d}t}{\mathrm{d}\tau}\frac{\mathrm{d}\mathbf{x}}{\mathrm{d}t}=\frac{t_f - t_0}{2}f(\mathbf{x,u})
\end{align}
The universal Birkhoff theory then allows us to transcribe \autoref{eq:opt_eqn} into the following NLP problem:
\begin{equation}\label{eq:opt_birkhoff}
\begin{aligned}
& \underset{\mathbf{u}}{\text{min}}
& & J = \frac{t_f-t_0}{2}\sum_{k=0}^N w_k||\mathbf{u}_k|| \\
& \text{subject to}
&\mathbf{X} &= \mathbf{X_0}\mathbf{1}^T+\frac{t_f-t_0}{2}\mathbf{BF}\\
& & R_{N} &= \mathbf{r}_{target}(t_f)\\
& & ||U_k|| &\leq u_{max}
\end{aligned}
\end{equation}
where $\mathbf{F}=f(\mathbf{X}(t), \mathbf{U}(t))$ is an $(N+1)\times6$ matrix of the dynamics at each node in physical time, and $w_k$ is the quadrature weighting of the $k$th node, given by:
\begin{align}
    w_k = \frac{2}{N(N+1)[P_{N}(x_k)]^2}
\end{align}
A step-by-step process to calculating $\mathbf{B}$ can be found in \cite{birkhoffuniversal}. 

However, this spacing is not suitable for all types of problems. Regions that exhibit rapid state variation may require a higher local node density to adequately resolve the dynamics with enough resolution. In manoeuvres such as debris collision avoidance manoeuvres, it may be desirable to cluster nodes near critical events, rather than using the symmetric distribution imposed by the standard LGL collocation scheme.

\subsection{Adaptive Node Placement via PDF}
Without loss of generality, $t_0=0$ is assumed throughout the remainder of this work. To allow LGL nodes to be dynamically clustered, the nodes will be clustered according to a normalized PDF $f_{X}=f_X(s;\mu_\text{PDF}, \sigma, \alpha)$ satisfying:
\begin{align}
    f_X(s)\geq 0, \hspace{2cm} \int_0^1 f_X(s) ds=1
\end{align}
which the user can design and tune using the mean $\mu_\text{PDF}\in[0, 1]$, standard deviation $\sigma$ and weighting $\alpha$. The weighting term $\alpha$ is only used if the user designs $f_X$ as a multimodal function of component densities $f_{X,1}, f_{X,2},...$, in which case $f_X=\sum_i \alpha_i f_{X,i}$ with $\sum_i \alpha_i = 1$ ensuring $f_X$ remains a valid normalized PDF. However, a single Gaussian suffices for the single-debris scenario considered in this work.

\begin{remark}
In practice, the Gaussian PDF $f_X$ is truncated to the domain $[0,1]$ and re-normalized such that $\int_0^1 f_X(s)\,ds = 1$. Additionally, a uniform floor $\varepsilon$ is blended into the distribution:
\begin{align}
    f_X(s) \leftarrow (1-\varepsilon)\,f_X(s) + \varepsilon
\end{align}
followed by re-normalization. This prevents $f_X(s) \approx 0$ near the boundaries of $[0,1]$, which would otherwise cause $\phi'(\tau) = \frac{t_f}{2f_X(t_{norm})}$ (see \autoref{eq:deriv}) to become arbitrarily large and introduce numerical ill-conditioning into the NLP. In this work, $\varepsilon = 0.05$ is used.
\end{remark}
Now, because the transformation from LGL nodes $\tau$ to physical time $t$ using a PDF is now non-linear, the transformation function changes from a simple affine transformation. We define the following intermediate values:
\begin{itemize}
    \item $\tau_\text{norm}$: These are the LGL nodes transformed from $[-1,1]$ to $[0,1]$ using an affine transformation:
    \begin{align}
        \tau_\text{norm} = \frac{\tau+1}{2}
    \end{align}
    \item $t_\text{norm}$: These are the final clustered time coordinates in the domain $[0, 1]$ such that:
    \begin{align}
        t = t_\text{norm}\cdot t_f
    \end{align}
\end{itemize}
We define $\tau_\text{norm}$ and $t_\text{norm}$ as intermediate steps as the mapping to be applied uses a cumulative distribution function (CDF), which has domain $[0, 1]$ and range $[0, 1]$. Additionally, because $t_f$ is free for the optimiser to solve for, the introduction of $t_\text{norm}$ decouples the node distribution from the final mission duration. To actually carry out the mapping between $\tau_\text{norm}$ and $t_\text{norm}$, define the CDF $F_X$:
\begin{align}
    F_X(s) = \int_{0}^s f_X(\zeta) d\zeta
\end{align}
The key here is that if $s=t_\text{norm}$ and $F_X(s)=\tau_\text{norm}$, then we can construct a mapping $t_\text{norm} = F_X^{-1}(\tau_\text{norm})$ which exactly follows the clustering behaviour of the PDF $f_X$. This is because when $f_X$ increases, this directly corresponds to a larger derivative in the CDF, causing the CDF to increase rapidly. Therefore, regions where $f_X$ attains larger values correspond to smaller increments in $t_\text{norm}$ for a fixed increment in $\tau_\text{norm}$, which is exactly the desired clustering effect. This is demonstrated in \autoref{fig:cdf_inv}, which shows an example of the remapping process using the inverse CDF from a PDF with a peak at $t_{norm}=0.3$ and a standard deviation of 0.05. Notice that the nodes cluster near $t_{norm}=0.3$, as desired, with the standard deviation determining the degree of clustering.
\begin{figure}[h!]
    \centering
    \includegraphics[width=\linewidth]{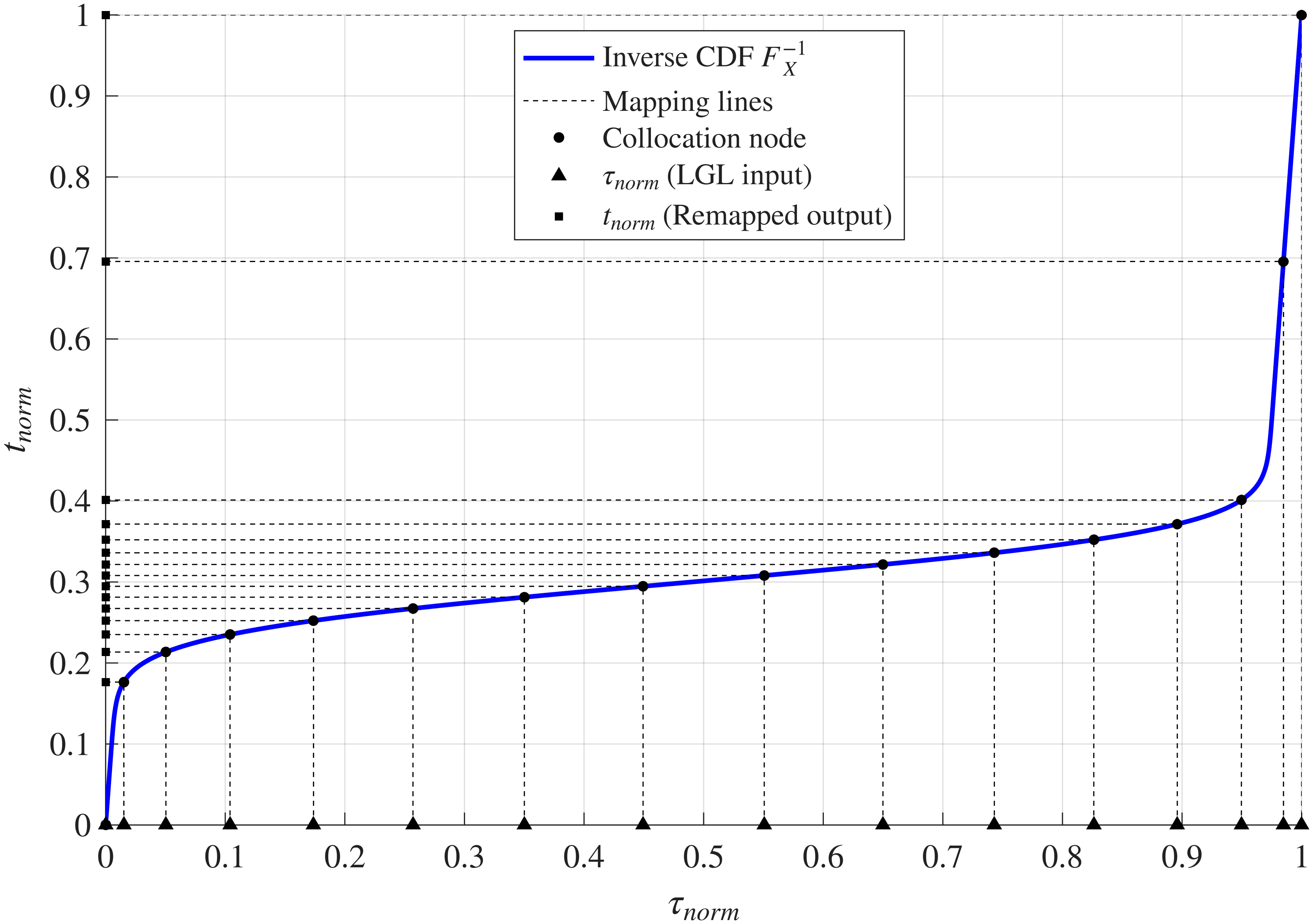}
    \caption{An example of remapping 16 LGL nodes using a PDF with a peak at $t_{norm}=0.3$ and $\sigma = 0.05$. }
    \label{fig:cdf_inv}
\end{figure}\\
In summary, the complete transformation is:
\begin{align}
    \tau \xrightarrow{\frac{\tau+1}{2}}\tau_{norm} \xrightarrow{F_X^{-1}} t_{norm} \xrightarrow{ t_f} t
\end{align}
which can be expressed by \autoref{eq:transform}. This is the \textbf{adaptive node placement transformation} equation.
\begin{align}\label{eq:transform}
    t = \phi(\tau) = t_f F_X^{-1}\left(\frac{\tau+1}{2} \right)
\end{align}
The proposed transformation preserves the monotonicity of the temporal mapping, since $f_X(s)>0$ implies $F_X$ is strictly increasing and therefore invertible. We also note that this transformation can be applied to a variety of optimisation scenarios.

\begin{proposition}
In the limit $\sigma \rightarrow \infty$, the adaptive node placement transformation \autoref{eq:transform} recovers the standard affine transformation \autoref{eq:affine}, and hence the original LGL collocation scheme.
\end{proposition}

\begin{proof}
As $\sigma \rightarrow \infty$, the Gaussian PDF becomes:
\begin{align}
    f_X(s) = \frac{1}{\sigma \sqrt{2\pi}}\exp{\left( \frac{-(s-\mu_\text{PDF})^2}{2\sigma^2}\right)}\rightarrow \frac{1}{\sigma \sqrt{2\pi}}
\end{align}
since the exponent $\frac{-(s-\mu_\text{PDF})^2}{2\sigma^2}\rightarrow0$ for any fixed $s$. After truncation and re-normalization to $[0, 1]$:
\begin{align}
    f_X(s) = \frac{1/\sigma\sqrt{2\pi}}{\int_0^1 1/\sigma\sqrt{2\pi}\,ds} = 1
\end{align}
So $f_X(s)\rightarrow 1$ uniformly on $[0, 1]$. The corresponding CDF is $F_X(s) = s$, and substituting into 
\eqref{eq:transform}:
\begin{align}
    \phi(\tau) = t_f F_X^{-1}\left(\frac{\tau+1}{2}\right) = 
    t_f\cdot\frac{\tau+1}{2}
\end{align}
which is the affine transformation \autoref{eq:affine} with $t_0=0$.
\end{proof}

\subsection{Changes to Birkhoff Approach Under New Mapping}
Under the new mapping, several constants in \autoref{eq:opt_birkhoff} are changed. Firstly, differentiate \autoref{eq:transform} to get:
\begin{align}\label{eq:deriv}
    \frac{\mathrm{d}t}{\mathrm{d}\tau} = \phi'(\tau) =  \frac{t_f}{2f_X(t_\text{norm})}
\end{align}
where we have used the inverse function theorem:
\begin{align}
    \frac{\mathrm{d}}{\mathrm{d}s}F_X^{-1}(s) = \frac{1}{F_X'(F_X^{-1}(s))} = \frac{1}{f_X(F_X^{-1}(s))}
\end{align}
Then applying the chain rule to the dynamics constraint gives:
\begin{align}
    X(\tau_k) &= X(\tau_0) + \int_{0}^{t_k} F \, dt \nonumber \\
              &= X(\tau_0) + t_f \int_{\tau_0}^{\tau_k} F \cdot \frac{1}{2f_X(t_\text{norm})} \, d\tau
\end{align}
By quadrature rules, this can be rewritten as:
\begin{align}\label{eq:birkhoff constraint}
    X(\tau_k) = X(\tau_0) + t_f \sum_{j=0}^NB_{kj} \frac
    {1}{2f_X(t_{\text{norm}, j})}F_j
\end{align}
or alternatively, in matrix form:
\begin{align}
    \mathbf{X} = \mathbf{1}\mathbf{X}_0 + t_f\,\mathbf{B} \,\text{diag}\left( \frac{1}{2f_X(t_\text{norm})}\right)\mathbf{F}
\end{align}
where $\mathbf{X},\mathbf{X}_0\in\mathbb{R}^{6\times(N+1)}$ and $\text{diag}\left( \frac{1}{2f_X(t_\text{norm})}\right)\in\mathbb{R}^{(N+1)\times(N+1)}$ is a diagonal scaling matrix that scales each row of $\mathbf{F}$ by $\phi'(\tau)$.
For the cost function, the adaptive node placement changes the quadrature to:
\begin{align}
    J &= \int_0^{t_f}||\mathbf{u}(t)|| dt \nonumber \\
    &=\int_{-1}^1 ||\mathbf{u}(t)||\phi'(\tau) d\tau \nonumber \\
    &= \sum_{k=0}^N w_k||\mathbf{u}(t_k)|| \frac{t_f}{2f_X(t_{\text{norm},k})}\label{eq:cost_new}
\end{align}
Thus, \autoref{eq:birkhoff constraint} and \autoref{eq:cost_new} are the adaptive-node equivalent of the dynamics constraints and cost function of \autoref{eq:opt_birkhoff} respectively.

\subsection{Collision Avoidance Manoeuvre Pipeline}
With the adaptive node placement algorithm, we are now in a position to design an autonomous collision avoidance system. We denote $\mathbf{x}$ the spacecraft state, $\mathbf{x_t}$ the target state, and $\mathbf{x_d}$ a debris state, from which the spacecraft must avoid on its way to the target. We also define $\rho_\text{det}$ as the detection range of the satellite, and $\rho_\text{KO}$ as the keep-out range where no debris are allowed to enter. 

The aim of this algorithm is to determine a nominal path so that $\mathbf{x}$ will rendezvous with $\mathbf{x_t}$, and replan accordingly if it detects and determines that $\mathbf{x_d}$ will enter the keep-out sphere. We will denote the nominal path calculation as Phase 1, and the replanning calculation as Phase 2. However, notice that if a debris object is detected, there is not yet enough information to determine where to cluster the nodes. Let $\mathbf{x_\text{det}}$ be the position of the spacecraft when it detects the debris. Hence, the replanning stage runs two solvers:
\begin{enumerate}
    \item Phase 2a: A universal Birkhoff solver will solve for a replanned path from $\mathbf{x_\text{det}}$ to $\mathbf{x_t}$, returning states and importantly, control history $\mathbf{u}$.
    \item Phase 2b: Next, the location of $||\mathbf{u}||_\text{max}$ in the temporal dimension is determined, and nodes are clustered at that timepoint. The adaptive solver is then run to obtain the replanned path.
\end{enumerate}
The algorithm is described in \hyperref[algo:adaptive_birkhoff]{Algorithm 1}.

\begin{algorithm}[h!]\label{algo:adaptive_birkhoff}
\caption{Adaptive Birkhoff Pseudospectral Collision Avoidance Pipeline}
\begin{algorithmic}[1]

\Require Initial spacecraft state $\mathbf{x_0}$, target state $\mathbf{x}_t$, debris state $\mathbf{x}_d(t)$, 
         detection range $\rho_\text{det}$, keep-out radius $\rho_\text{KO}$, 
         number of nodes $N$, thrust limit $u_\text{max}$
\Ensure  Fuel-optimal trajectory $\mathbf{x}(t)$ satisfying keep-out constraint

\Statex
\Statex
\textbf{Phase 1: Nominal Trajectory}
\State Compute LGL nodes $\{\tau_k\}_{k=0}^N$ and apply affine transformation \eqref{eq:affine}
\State Solve Birkhoff NLP \eqref{eq:opt_birkhoff} for nominal trajectory $\mathbf{x}^*(t)$, control $\mathbf{u}^*(t)$ and rendezvous time $t^*$
\State Propagate spacecraft state $\mathbf{x}(t)$ and debris state $\mathbf{x}_d(t)$ forward
\Statex
\Statex
\textbf{Phase 2: Conjunction Detection and Replanning}
\If{$\|\mathbf{r}(t) - \mathbf{r}_d(t)\| \leq \rho_\text{det}$ for any $t \in [t_0, t_f]$}
    \State Identify time of debris detection $t_\text{det}$
    \Statex
    \Statex
    \textbf{Phase 2a: Flat-PDF Replan}
    \State Set $f_X(s) = 1$ (uniform distribution, recovers standard LGL)
    \State Apply transformation \eqref{eq:transform} and solve modified Birkhoff NLP \eqref{eq:birkhoff constraint} subject to:
    \begin{equation*}
        \|\mathbf{r}_k - \mathbf{r}_d(t_k)\| \geq \rho_\text{KO}, \quad k = 0, \ldots, N
    \end{equation*}
    \Statex
    \textbf{Phase 2b: Adaptive-PDF Replan}
    \State Identify peak thrust time $t_{peak} = \arg\max_t \|\mathbf{u}^*_{2a}(t)\|$ from Phase 2a solution
\State Set $\mu_\text{PDF} = t_{peak}/t^*$, construct Gaussian PDF $f_X(s; \mu_\text{PDF}, \sigma)$
    \State Apply transformation \eqref{eq:transform} and solve modified Birkhoff NLP \eqref{eq:birkhoff constraint} subject to:
    \begin{align*}
        \|\mathbf{r}_k - \mathbf{r}_d(t_k)\| &\geq \rho_\text{KO}, \quad k = 0, \ldots, N\\
        \text{time of flight } t_\text{replan} &= t^* - t_\text{det}
    \end{align*}
    \State Follow replanned trajectory $\mathbf{x}^*(t)$ to $t_f$
\Else
    \State Follow nominal trajectory $\mathbf{x}^*(t)$ to $t_f$
\EndIf

\end{algorithmic}
\end{algorithm}

\newpage
\section{Results and Discussions}
The adaptive node pipeline was applied with the initial states and variables of \autoref{tab:init_cond}. These initial states are based on the 9th Global Trajectory Optimisation Competition (GTOC) \cite{gtoc} which provides a large catalogue of low-Earth orbit (LEO) debris objects with realistic orbital parameters, and hence offers a physically representative test environment for the proposed method. The detection range $\rho_\text{det}$ and keep-out radius $\rho_\text{KO}$ are selected such that the conjunction is triggered and the keep-out constraint is active for the given scenario, allowing the performance of the proposed method to be meaningfully evaluated.

While these values exceed typical operational thresholds, which are on the order of kilometres for keep-out zones and reported miss distances \cite{detection_range, detection2}, they are selected to ensure a well-defined conjunction geometry for the purposes of demonstrating the adaptive node placement framework. All simulations were run using MATLAB 2025b on a 2.20 GHz Intel Core Ultra 7 265H 64-bit CPU with 32 GB RAM, and the NLPs were solved using IPOPT (Interior Point optimiser) \cite{ipopt}, interfaced through CasADi \cite{casadi} for automatic differentiation and problem formulation. The initial conditions for simulation are shown in \autoref{tab:init_cond}.

\begin{table}[h!]
\centering
\caption{Initial conditions for simulation (Earth-Centered Inertial Frame)}
\rowcolors{2}{gray!15}{white}
\begin{tabular}{lccc}
\hline
\textbf{Parameter} & $x$ & $y$ & $z$ \\
\hline
$\mathbf{r}_0$ (km) & 136.094 & -875.503 & -7078.351 \\
$\mathbf{v}_0$ (km/s) & -7.321 & 1.574 & -0.0643 \\
$\mathbf{r}_{d_0}$ (km) & -5637.474 & 1247.190 & 4123.303 \\
$\mathbf{v}_{d_0}$ (km/s) & -5.604 & 0.698 & -3.654 \\
$\mathbf{r}_{t_0}$ (km) & -2449.412 & 3986.821 & -5440.310 \\
$\mathbf{v}_{t_0}$ (km/s) & 4.226 & -3.883 & -4.743 \\
\hline
\textbf{Parameter} & & & \\
\hline
$\rho_\text{det}$ (km) & & 1000 & \\
$\rho_\text{KO}$ (km) & & 300 & \\
\hline
\end{tabular}
\label{tab:init_cond}
\end{table}

\subsection{Collision Avoidance Simulation Results}
All phases of the simulation (including the adaptive-node phase) were run with 50 LGL nodes, which were distributed as shown in \autoref{fig:node_dist}. The fuel cost $\Delta v$ was then calculated for each portion of the trajectory, and the key results are shown in \autoref{tab:results_main}.

\begin{figure}[h!]
    \centering
    \includegraphics[width=0.8\linewidth]{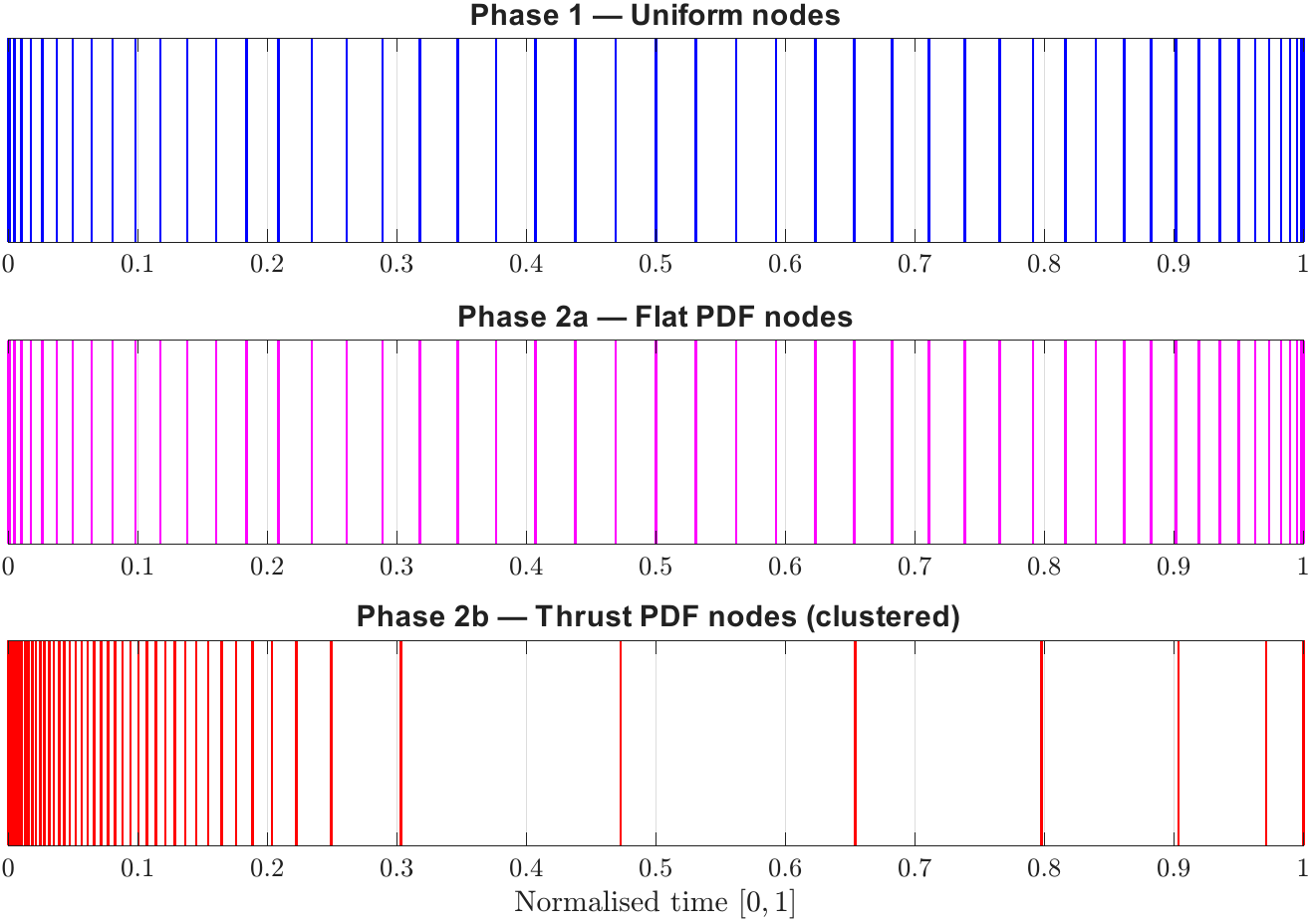}
    \caption{The node distributions for each phase of the algorithm}
    \label{fig:node_dist}
\end{figure}

\begin{table}[h!]
\centering
\caption{Results from the collision avoidance simulation}
\begin{tabular}{lcc}
\toprule
\multicolumn{3}{c}{\textbf{Nominal Trajectory Results}} \\
\midrule
\rowcolor{gray!15}
\multicolumn{2}{l}{Final rendezvous time $t^*$ (s)} & 3697   \\
\multicolumn{2}{l}{Detection time $t_\text{det}$ (s)}    & 1225.4 \\
\rowcolor{gray!15}
\multicolumn{2}{l}{Nominal total $\Delta v$ (km/s)} & 2.596 \\
\midrule
\textbf{Fuel budget} & \textbf{$\Delta v$ (km/s)} & \textbf{Extra cost vs. nominal (km/s)} \\
\midrule
\rowcolor{gray!15}
Phase 1, before $t_\text{det}$           & 1.085 & --    \\
Phase 1, after $t_\text{det}$              & 1.511 & --    \\
\rowcolor{gray!15}
Phase 2a, flat PDF                  & 1.793 & 0.282 \\
Phase 2b, adaptive PDF              & 5.950 & 4.439 \\
\bottomrule
\end{tabular}
\label{tab:results_main}
\end{table}

At first, it looks as if the flat-PDF replanning algorithm (ie. universal Birkhoff theory) outperforms the adaptive node algorithm in terms of pure fuel consumption. However, the separation distance between the spacecraft and the debris was calculated, as shown in \autoref{fig:sep_hist}.
\begin{figure}[h!]
    \centering
    \includegraphics[width=\linewidth]{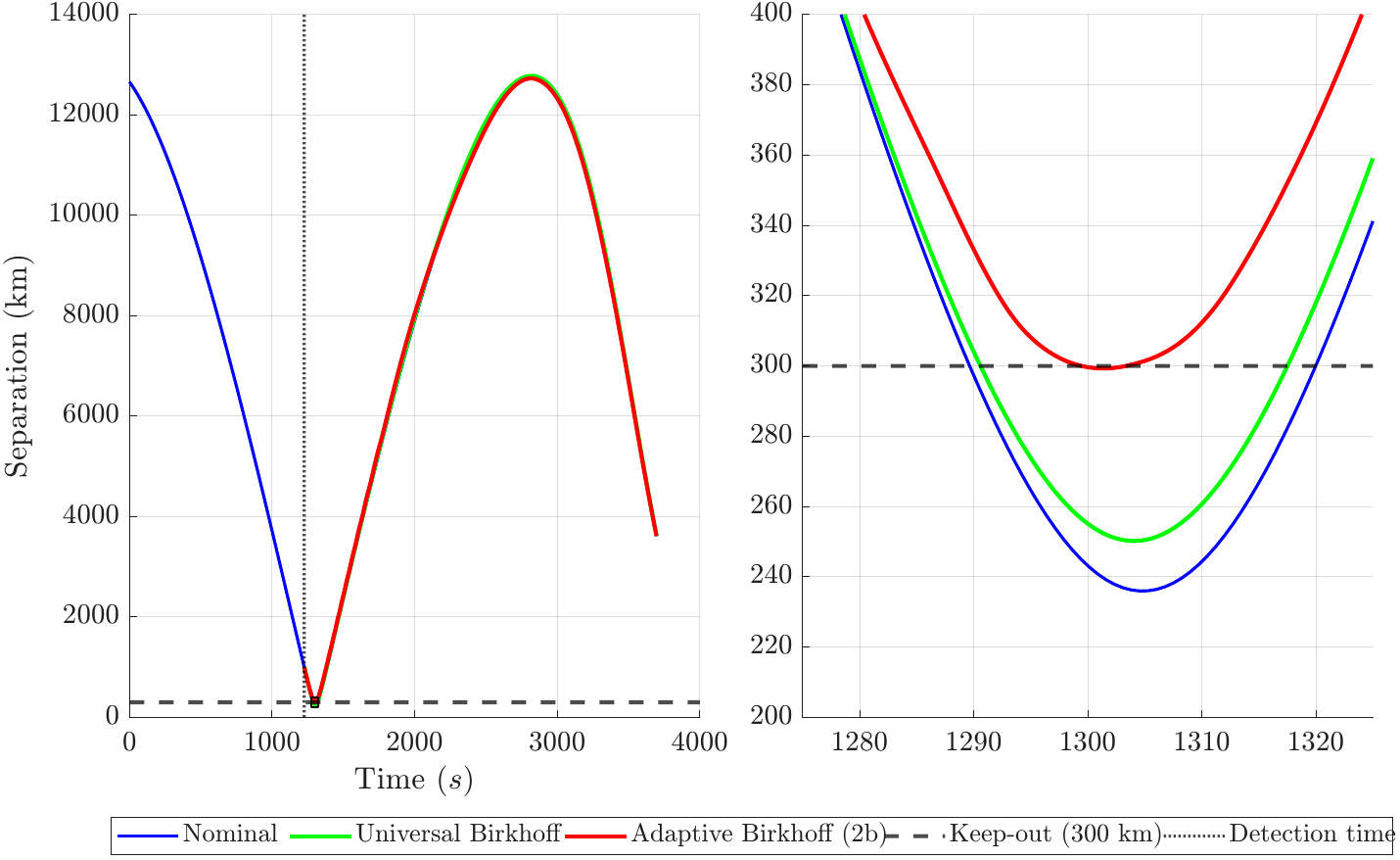}
    \caption{Left: The separation history between the spacecraft and debris for both replanning phases. Right: Zoomed-in version near the keep-out region.}
    \label{fig:sep_hist}
\end{figure}\\
\autoref{fig:sep_hist} highlights how while the nominal trajectory violates the keep-out constraints as expected, Phase 2a (replanning using the universal Birkhoff method) also violates the keep-out constraints. Both the standard and adaptive methods satisfy the keep-out constraint exactly at the discrete collocation nodes, in line with constraint enforcement in pseudospectral collocation. However, the densely interpolated reconstruction highlights a substantial difference in the violation occurring between nodes, with a $49.80$ km violation for the standard method, compared to $0.70$ km for the adaptive method - a reduction of 98.6\%. Relative to the keep-out radius of $\rho_{KO} = 300$km, this corresponds to a $16.6\%$ intrusion for the standard method, in contrast to a comparatively negligible $0.23\%$ for the adaptive method.

This disparity reflects the advantages of the concentration of nodes near the keep-out region by the adaptive method, which sharply narrows the unconstrained inter-node gap. Therefore, the adaptive node algorithm correctly identifies where resolution is needed and produces a genuinely feasible manoeuvre, while the flat-PDF produces a cheap solution that only satisfies the keep-out condition at the collocation nodes, with a substantial violation occurring between them. For debris collision avoidance, this is a significant caveat of the universal Birkhoff theory and clearly shows the advantages of having a flexible node placement.

The high $\Delta v$ of the adaptive formulation requires further examination. Given the short remaining arc of $t_{replan} = 2471.6$ seconds after detection, the optimiser must execute an aggressive early deflection manoeuvre to guarantee clearance throughout the conjunction window, which leaves very little of the arc left to recover a fuel-efficient approach to the target. This is shown as a sustained thrust saturation near the start of the replanning stage, as shown in \autoref{fig:control_mag}. 
\begin{figure}[h!]
    \centering
    \includegraphics[width=0.8\linewidth]{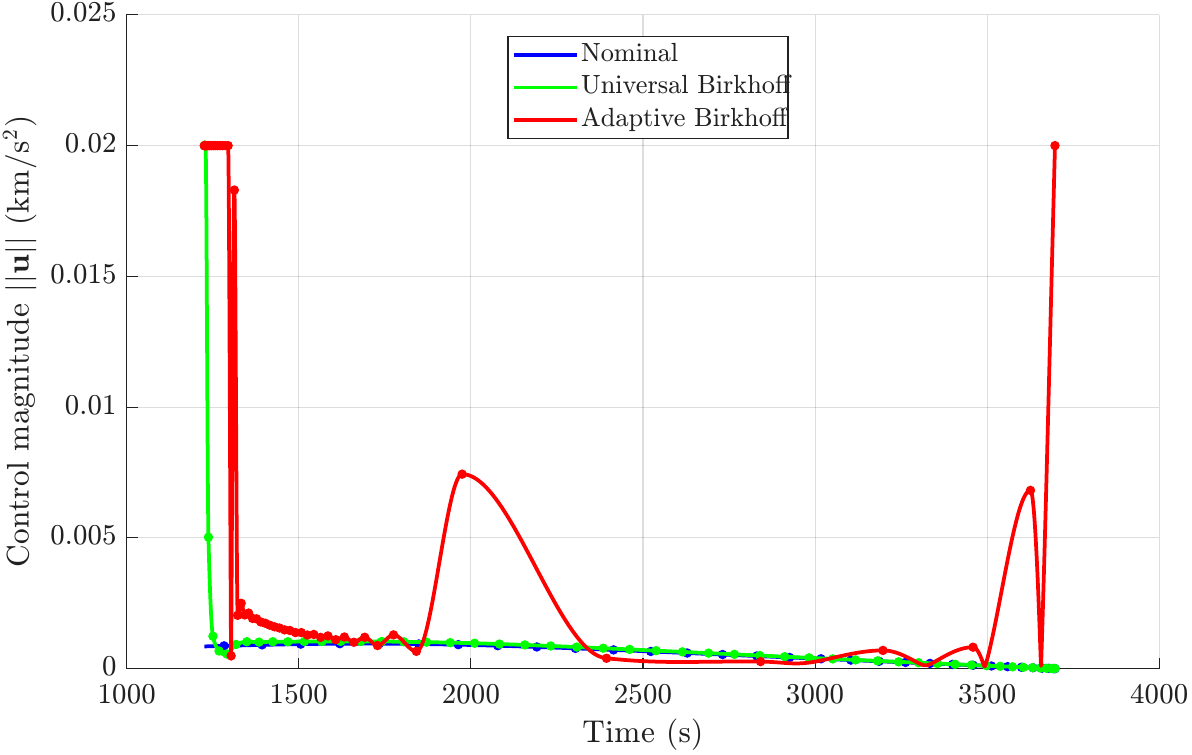}
    \caption{Comparison of the control magnitudes from each replanning algorithm. Dots mark the location of the discrete collocation nodes.}
    \label{fig:control_mag}
\end{figure}\\
A secondary saturation spike in the end of the adaptive formulation is also visible. As the PDF concentrates nodes around $t_\text{peak}$, it leaves the terminal region comparatively sparse and hence the optimiser has little resolution left to shape a smooth final approach, instead relying on a saturated control correction at the last few nodes to meet the hard terminal rendezvous constraint. This can be mitigated with a multimodal $f_X$ that adds a secondary component near $t_f$, a larger $\sigma$ value (at the cost of resolution near the detection window), or by increasing the node count (though this comes with a caveat, see \hyperref[sec:node_refinement]{Section 3.2}.) We also note that warm-starting with the universal Birkhoff solution does not change the outcome, suggesting the result is robust to initialization, and since the same solver and tolerances were used across all three stages, this value reflects a locally optimal, feasible solution under our problem formulation. The cost difference therefore reflects the price of genuine constraint satisfaction throughout the conjunction window rather than only at the collocation nodes. 

\subsection{Limitations of LGL Node Refinement} \label{sec:node_refinement}
A natural follow-up question is whether the constraint violation in Phase 2a can be resolved simply by increasing the number of collocation nodes, rather than through adaptive placement, in order to reduce the inter-node gap and hence the inter-node constraint violations. Indeed, \cite{ross_birkhoff_part2} highlighted how the universal formulation could handle 700 grid points for a challenging low-thrust orbital transfer problem. \autoref{fig:node_converge} shows the minimum separation distance as a function of node count $N$ for the flat-PDF replanning algorithm.
\begin{figure}[h!]
    \centering
    \includegraphics[width=\linewidth]{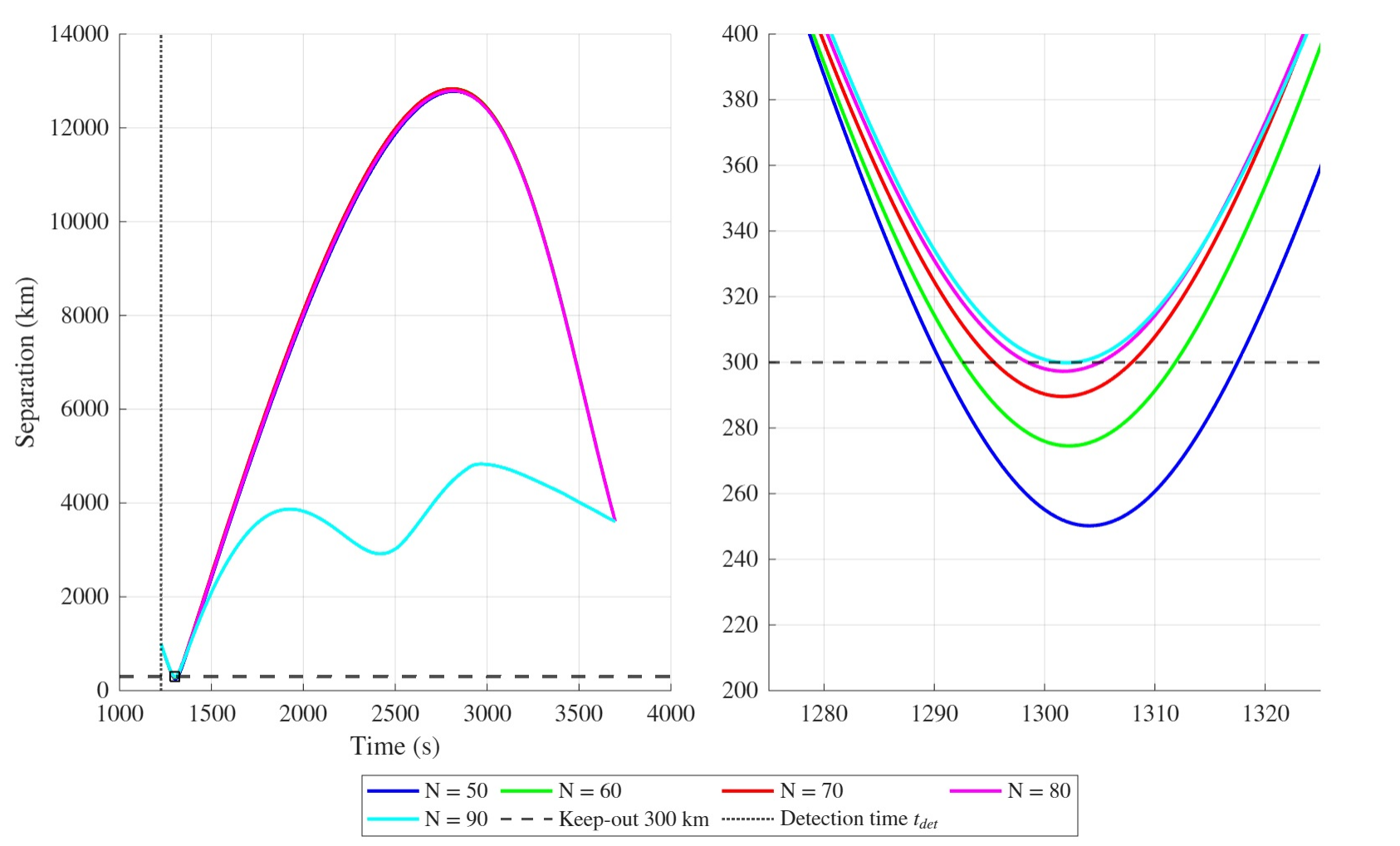}
    \caption{Left: Separation history as a function of the number of nodes $N$. Right: zoomed-in version near the keep-out region.}
    \label{fig:node_converge}
\end{figure}\\
As \autoref{fig:node_converge} shows, at $N=80$ the keep-out constraint performance has improved but still does not reach $\rho_{KO}=300$ km, and at $N=90$, the solution fails to converge within the 2000 iteration limit, producing a solution that meets the constraints, but is not optimised - at $N=90$, the fuel consumption was $\Delta v = 42.00$ km/s. Even when the maximum number of NLP iterations was increased to 10,000, an optimised result was not able to be reached, suggesting that the corresponding NLP becomes increasingly difficult to solve numerically as $N$ increases. Furthermore, warm-starting from the $N=80$ solution also did not help, and instead resulted in a worse fuel consumption of $\Delta v =50.77$ km/s. 

To investigate exactly where the breakdown occurred, a finer sweep of $N$ was run for $N=81$ to $85$. This is shown in \autoref{fig:zoomed_in_sep}.
\begin{figure}[h!]
    \centering
    \includegraphics[width=\linewidth]{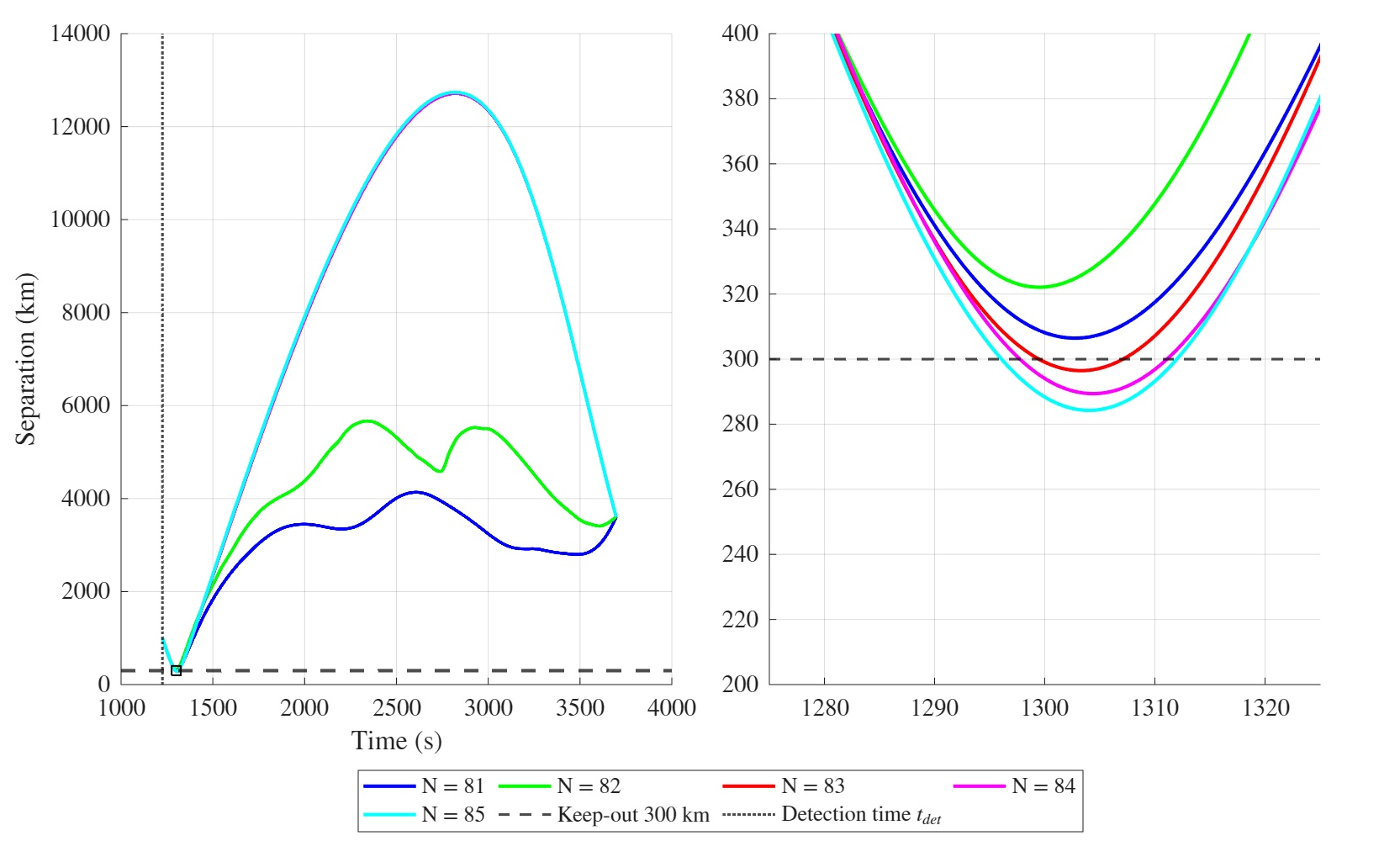}
    \caption{Left: Separation history with finer resolution near $N=80$. Right: zoomed-in version near the keep-out region.}
    \label{fig:zoomed_in_sep}
\end{figure}\\
\autoref{fig:zoomed_in_sep} highlights a clear transition in solver behaviour at $N=81$, which is the first node count to show this unstable behaviour. However, $N=81$ is also the first node size for which the keep-out constraint is actually satisfied by the universal Birkhoff method. This suggests that the observed deterioration in convergence is strongly related to a change in the optimisation problem as sufficient discretisation resolution is reached to satisfy the keep-out constraint. This behaviour was also seen at $N=82$, but surprisingly not at $N=83$ or $N=84$, which showed separation behaviour similar to the original, nominal path. This highlights a crucial fact about the solver behaviour under the universal Birkhoff formulation - the solver instability is not a result of larger $N$, but of another factor. 

Several potential causes for this irregular behaviour were investigated. First, the condition numbers of the Birkhoff matrix $[\mathbf{I, -B}]$ returned a clean $\mathcal{O}(1)$ curve in perfect alignment with \cite{ross_birkhoff_part1}, ruling out increasing ill-conditioning of the Birkhoff discretisation matrix as the cause. Next, the full equality-constraint Jacobian $J_c$, which includes both the Birkhoff collocation defects and the non-linear dynamics, was evaluated at two points via singular value decomposition at the initial guess and the converged (or best-so-far) solution for each $N$, to obtain its condition number $\kappa(J_c)$. To test for a linear independence constraint qualification (LICQ) breakdown, the smallest singular value of the Jacobian rows for constraints near the 300 km bound, $\sigma_{\min}(J_{\text{active}})$, was also computed. Both quantities were insensitive to $N$, with $\kappa(J_c)$ remaining within a narrow band with no monotonic trend, and $\sigma_{\min}(J_{\text{active}})$ being constant at 600 for all tested $N$ values. This rules out both intrinsic conditioning of the discretised NLP and degradation in the independence of the active keep-out constraint gradients as the cause of the $N=90$ failure.

This points to one remaining factor, which is that the misbehaviour must be occurring at some point in the middle of the solve. The issue is that for $N= 81$, the scaled dual infeasibility measure (\verb|inf_du| in IPOPT) fails to decrease towards the solver tolerance even after 10,000 iterations, remaining at magnitudes of $10^6$ to $10^{13}$. This indicates that despite meeting the constraints, the solution is nowhere near a Karush-Kuhn-Tucker (KKT) local optimum since the first-order optimality condition is being violated by orders of magnitude. IPOPT can address this using a restoration phase, where the solver temporarily abandons the primal-dual subproblem and instead minimises constraint violation to find a new iteration closer to feasibility that the regular optimisation can resume from. In this run, the restorations were triggered and successfully re-established feasibility, but the subsequent iterations drove \verb|inf_du| back to large values. 

Therefore, we investigate whether the solver behaviour is related to the placement of the collocation nodes relative to the point of closest approach. Let $t_\text{critical}$ denote the time at which the spacecraft reaches its minimum separation from the debris:
\begin{align}
    t_\text{critical} = \arg \min_t d(t), \quad d(t)=||\mathbf{r}(t) - \mathbf{r}_\text{deb}(t)||
\end{align}
For a discretisation with $N$ nodes, we define the temporal distance of the closest node to this point as $\text{sep}(N)$:
\begin{align}\label{eq:sepn}
    \text{sep}(N) := \min_{k\in \{1, 2, ..., N \}} |t_k^{(N)} - t_\text{critical}|
\end{align}
where $t_k^{(N)}$ is the temporal location of the $k$th node in a discretisation of size $N$. Since the physical closest-time approach is determined by the underlying spacecraft and debris trajectories, $t_\text{critical}$ should be approximately independent of the discretisation, since the finer the resolution, the more the solution should converge to an optimised avoidance trajectory. In practice, we re-evaluate $t_\text{critical}$ for each $N$ value for the purposes of determining the reason for solver divergence. The hypothesis is that a node located very close to $t_\text{critical}$ may produce a more difficult discretised constraint, especially because this is the point where the keep-out constraint is most restrictive and the separation has a stationary point. The results are shown in \autoref{tab:tcrit}.

\begin{table}[h!]
\centering
\caption{The temporal distance of the closest node to $t_\text{critical}$, and the corresponding solver behaviour.}
\label{tab:tcrit}
\begin{tabular}{ccc}
\toprule
$N$ & $\text{sep}(N)$ (s) & Behaviour \\
\midrule
79 & 5.44 & smooth parabolic separation  \\ \rowcolor{gray!15}
80 & 3.44 & smooth parabolic separation  \\ 
81 & 1.52 &  irregular separation \\\rowcolor{gray!15}
82 & 1.32 &  irregular separation  \\ 
83 & 3.60 &  smooth parabolic separation  \\\rowcolor{gray!15}
84 & 6.81 & smooth parabolic separation \\ 
85 & 7.94 & smooth parabolic separation  \\ \rowcolor{gray!15}
\bottomrule
\end{tabular}
\end{table}

As \autoref{tab:tcrit} shows, irregular behaviour occurs when $\text{sep}(N) \lessapprox 2$ sec, while all cases with larger temporal separation exhibit smooth parabolic behaviour. This suggests an empirical threshold of approximately 2 seconds and importantly, the transition is not monotonic with $N$, indicating that the deterioration is associated with the particular LGL node distribution rather than the magnitude of $N$. This also empirically explains the irregular separation at $N=90$ seen in \autoref{fig:node_converge}, which had $\text{sep}(N)=1.21$ seconds.

\subsubsection{Relationship Between $\text{sep}(N)$ and Keep-out Violations}
The pattern in \autoref{tab:tcrit} can be explained by observing that the separation $d(t)$ is a smooth function of time with a genuine local minimum (a turning point) at $t_\text{critical}$, which is determined by the underlying chaser and debris kinematics.

However, the NLP enforces the keep-out constraint only at the discrete LGL nodes. If $t_\text{critical}$ does not coincide with a collocation node, then it necessarily lies strictly between two adjacent nodes (call them node $k$ and node $k+1$). In particular, when both surrounding nodes are optimised to the keep out boundary:
\begin{align}
    d(t_k)=d(t_{k+1})=\rho_{KO}
\end{align}
By Rolle's theorem, it is guaranteed that there exists a stationary point between them. Since $t_\text{critical}$ is a strict local minimum, it follows that:
\begin{align}
    d(t_\text{critical}) < d(t_k)=d(t_{k+1})=\rho_{KO}
\end{align}
and hence a continuous-time keep-out violation is unavoidable. Using this idea, we can then derive an expression that predicts the keep-out violation. Note that:
\begin{align}
    d(t) =||\mathbf{r}_\text{rel}(t)||=\sqrt{\mathbf{r}_\text{rel}(t)\cdot \mathbf{r}_\text{rel}(t)}\implies\dot{d}(t) = \frac{\mathbf{r}_{\mathrm{rel}}(t)^\mathsf{T}\mathbf{v}_{\mathrm{rel}}(t)}{||\mathbf{r}_{\mathrm{rel}}(t)||} = \hat{\mathbf{r}}_\text{rel}\cdot \mathbf{v}_\text{rel}
\end{align}
and since it has a local minimum at $t_\text{critical}$:
\begin{align}
    \dot{d}(t_\text{critical}) = 0 \implies \hat{\mathbf{r}}_\text{rel}(t_\text{critical})\cdot \mathbf{v}_\text{rel}(t_\text{critical}) = 0.
\end{align}
Decomposing $\mathbf{v}_\text{rel}$ into components parallel and perpendicular to $\hat{\mathbf{r}}_\text{rel}$,
\begin{align}
    \mathbf{v}_\text{rel} =
    \underbrace{\left(\hat{\mathbf{r}}_\text{rel}\cdot\mathbf{v}_\text{rel}\right)\hat{\mathbf{r}}_\text{rel}}_{\mathbf{v}_{\text{rel},\parallel}}
    +
    \underbrace{\left[\mathbf{v}_\text{rel} - \left(\hat{\mathbf{r}}_\text{rel}\cdot\mathbf{v}_\text{rel}\right)\hat{\mathbf{r}}_\text{rel}\right]}_{\mathbf{v}_{\text{rel},\perp}},
\end{align}
the vanishing dot product above gives $\mathbf{v}_{\text{rel},\parallel}(t_\text{critical}) = \mathbf{0}$, and therefore:
\begin{align}
    \mathbf{v}_\text{rel}(t_\text{critical}) = \mathbf{v}_{\text{rel},\perp}(t_\text{critical}) =: v_\perp\,\hat{\mathbf{e}}_\perp,
\end{align}
i.e.\ the relative velocity at closest approach is purely transverse, with magnitude $v_\perp := |\mathbf{v}_\text{rel}(t_\text{critical})|$. Over a short window around $t_\text{critical}$, the relative motion can be approximated as straight-line at constant velocity, giving the classical closest-approach relation:
\begin{align}
    d(t)^2 \approx d_{\min}^2 + v_\perp^2(t - t_\text{critical})^2. \label{eq:closest_approach}
\end{align}

Now, we want to derive an expression for $d_{min}$, the minimum separation (which occurs at $t_\text{critical}$). Since we know that inter-node violations occur when the nearby nodes are optimised to $\rho_{KO}$, the node closest to $t_\text{critical}$ (call it node $k$) is optimised to exactly $d(t_k)=\rho_{KO}$. Substituting $t=t_k$, $d(t_k)=\rho_\text{KO}$, and $|t_k - t_\text{critical}| = \text{sep}(N)$ into \eqref{eq:closest_approach} and solving for the true minimum separation at the turning point gives:
\begin{align}
    d_{\min} &= \sqrt{\rho_{KO}^2 - v_\perp^2\,\text{sep}(N)^2}\label{eq:dmin}\\
   \therefore \text{violation}(N) &= \rho_{KO} - d_{min}\nonumber \\
    &=\rho_{KO} -\sqrt{\rho_{KO}^2 - v_\perp^2\,\text{sep}(N)^2} \nonumber \\
    &\approx  
    \frac{v_\perp^2\,\text{sep}(N)^2}{2\rho_{KO}}, \label{eq:violation_pred}
\end{align}
where the final expression follows from a first-order Taylor expansion of $\sqrt{1-x}$ about $x=0$, with $x=v_\perp^2\text{sep}(N)^2/\rho_\text{KO}^2$. This approximation is valid when $v_\perp\text{sep}(N)\ll\rho_\text{KO}$. Note that \autoref{eq:dmin} predicts that the further the nearest collocation node is from the true closest-approach time (ie. higher $\text{sep}(N)$), the smaller the minimum separation $d_{min}$ is (ie. the spacecraft and debris are closer), and hence the larger inter-node violation by \autoref{eq:violation_pred}, which is exactly what was observed.

Reliable local convergence of interior-point methods depends on appropriate regularity conditions, such as the LICQ, strict complementarity, and second-order sufficient conditions \cite[Sec. 19.8]{nocedal_wright_2006}. As \autoref{eq:violation_pred} shows, a smaller $\text{sep}(N)$ produces a smaller predicted inter-node violation, bringing the constraint closer to being marginally active. In this regime, the associated Lagrange multiplier may approach zero, weakening strict complementarity and potentially making the primal-dual system more numerically sensitive. This provides a possible explanation for the irregular behaviour and repeated restoration-phase cycling observed previously.

\subsubsection{Implications for Adaptive Node Placement}

The adaptive method, in contrast to the universal Birkhoff results, achieves a locally optimal solution much faster than the constraint-violating $N=80$ flat-PDF replan, and with a lower node count at $N=50$. Additionally, its violation of 0.70 km is an improvement on the 2.72 km violation generated by the universal formulation with $N=80$ - this can be attributed to the closer spacing of the adaptive nodes, which reduces the temporal interval over which violations from the constraint boundary can occur. Therefore, the results of the adaptive method highlight an important limitation of the universal formulation for problems with localised constraint activity - increasing the number of nodes does not necessarily improve continuous-time feasibility, since the redistribution of nodes can produce unfavourable placements relative to the critical region while still leaving substantial inter-node gaps. The adaptive discretisation instead focuses the resolution in the region of greatest trajectory activity, which substantially reduces the inter-node violation at lower node counts. 

However, this clustering does not inherently eliminate the numerical difficulties associated with nodes located very close to the point of closest approach - indeed, increasing the adaptive node count can reproduce the same irregular behaviour (for example, using $N=80$ nodes with the adaptive method also produces irregular behaviour). This creates competing effects, where clustering the nodes gives smaller inter-node gaps, but also potentially smaller $\text{sep}(N)$ as the nodes get closer to $t_{critical}$ which introduces the numerical difficulties described previously. Therefore, although the adaptive method cannot guarantee solver stability for arbitrarily large node counts, its principal advantage is that it can achieve substantially improved continuous-time constraint satisfaction at a lower node count and computational cost. In this sense, the adaptive method provides a more efficient discretisation of the trajectory, allowing localised constraint activity to be resolved before the numerical difficulties associated with excessive proximity to $t_\text{critical}$ become dominant.

Another practical consequence of increasing nodes is wall time which increases nonlinearly with $N$, from 6.8 seconds at $N=50$ to 113.3 seconds at $N=80$, as shown in \autoref{fig:wall_time}.
\begin{figure}[h!]
    \centering
    \includegraphics[width=0.55\linewidth]{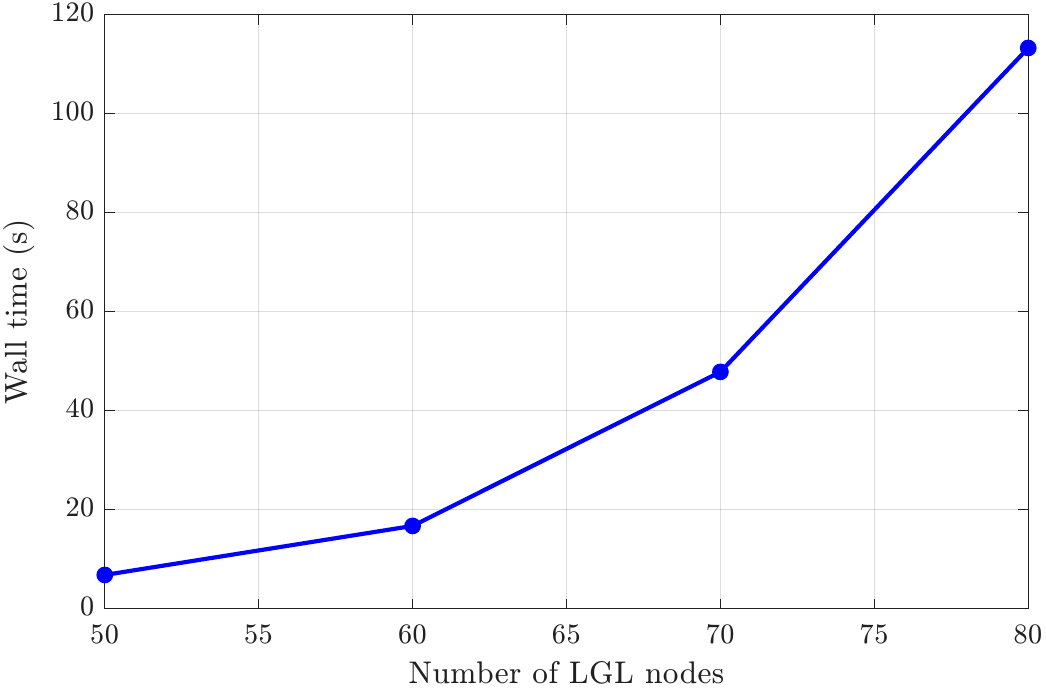}
    \caption{Cumulative wall time as a function of the number of nodes $N$ with the universal method.}
    \label{fig:wall_time}
\end{figure}\\
In comparison, the adaptive method with $N=50$ nodes was solved in 18.8 seconds, which far outperforms $N=80$ while also being more accurate in terms of inter-node violations. Hence, having fewer nodes to optimise for not only provides both numerical advantages but also computational advantages.

\subsection{Comparison to MPC}
MPC is a well-established existing method for collision avoidance control \cite{park, bashnick, xwang}, so we compare the performance of the adaptive node algorithm to MPC. For this paper, a receding-horizon terminal MPC, where the algorithm is free to choose its own path as long as terminal constraints are met, was used. 

To ensure a fair comparison, the total flight time of the MPC replanning manoeuvre was fixed to $t_{replan}$, matching that of the standard and adaptive Birkhoff methods so that the comparison is isolated to differences in trajectory-shaping. The MPC replans at every timestep by solving a finite-horizon optimal control problem over a prediction horizon $T_p$, which is clamped near the end of the manoeuvre to avoid extending beyond the fixed terminal time. The prediction horizon is discretized into $N_{mpc} = \frac{T_p}{T_s}$ steps of length $T_s$ seconds. At each replanning instance, only the first control input of the solved sequence is applied before the horizon shifts forward and the problem is re-solved with updated state information. The objective function per replanning instance is:
\begin{equation}\label{eq:opt_mpc}
\begin{aligned}
& \underset{\mathbf{u}}{\text{min}}
&  J &= \underbrace{\sum_{k=1}^{N_{mpc}} \frac{T_{p}}{N_\text{mpc}} ||\mathbf{u_k}||}_\text{fuel cost} 
       + \underbrace{\alpha_\text{KO}\sum_{k=1}^{N_\text{mpc}+1}\left[\text{max}\left(0, \rho_\text{KO}^2-||\mathbf{r}_k - \mathbf{r_{d,k}}||^2 \right)\right]^2}_\text{soft keep-out penalty} \\
& & & \quad + \underbrace{\alpha_k||\mathbf{r_{N_\text{mpc}+1}} - \mathbf{r_t}||^2}_\text{scaled terminal cost} \\
& \text{subject to}
& \mathbf{x}_{k+1} &= x_{k} + \int_{t_k}^{t_{k+1}} f(\mathbf{x}_k, \mathbf{u}_k) dt\\
& & ||\mathbf{u}_k|| &\leq u_\text{max}\\
& & \mathbf{x}_1 &= \mathbf{x_0}\\
& & \mathbf{v}_{{N}_\text{mpc}+1} ^T\mathbf{v}_\text{target} &> 0
\end{aligned}
\end{equation}
where $\alpha_\text{KO}$ and $\alpha_k$ are tuning parameters for the soft keep-out penalty and scaled terminal cost respectively, and the discretization step $\frac{T_p}{N_\text{mpc}}$ is nominally equal to $T_s$, except near the end of the manoeuvre where the horizon $T_p$ is clamped, temporarily reducing the step length. While $\alpha_\text{KO}$ is constant across each solve, the terminal weight $\alpha_k$ increases as the horizon approaches the fixed terminal time, progressively prioritizing convergence to the target as less time remains to correct any residual error. Note also that we sum up to $N_\text{mpc}+1$ for state-related terms as there are $N_\text{mpc}+1$ state nodes, but we only sum up to $N_\text{mpc}$ for the fuel cost as there are $N_\text{mpc}$ control nodes (we do not apply a control at the last state). The last constraint $\mathbf{v}_{{N}_\text{mpc}+1} ^T\mathbf{v}_\text{target} > 0$ is added to enforce alignment between the terminal spacecraft velocity and the target velocity, which prevents convergence to solutions that use unfavourable phasing geometry. Additionally, the method of using a soft keep-out penalty in the objective was chosen over setting it as a constraint, as enforcing it directly was found to produce persistent solver failures. Specifically, IPOPT repeatedly failed to converge, which was diagnosed as the solution grazing the constraint boundary of $\rho_{KO}$ within numerical tolerance during the replanning process.

To determine a representative MPC solution to compare the standard and adaptive results against, the horizon length $N_{mpc}$ was swept from 40 to 80 steps to identify the best-performing MPC configuration, as shown in \autoref{fig:compare_mpc}.

\begin{figure}[h!]
    \centering
    \includegraphics[width=0.65\linewidth]{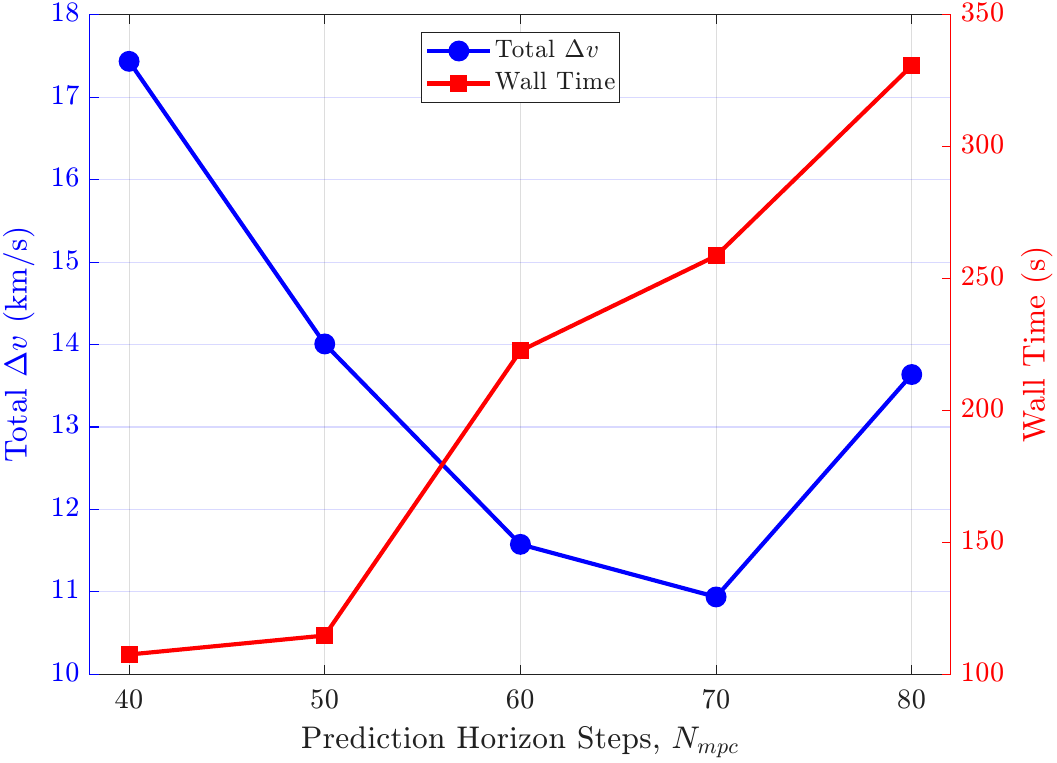}
    \caption{Effects of increasing horizon steps on $\Delta v$ and wall time}
    \label{fig:compare_mpc}
\end{figure}
As \autoref{fig:compare_mpc} shows, the $\Delta v$ initially decreases with prediction horizon steps, with a $37.4\%$ reduction in $\Delta v$ from $17.4$ km/s with $N_{mpc}=40$ to $\Delta v=10.9$ km/s with $N_{mpc}=70$. However, $N_{mpc}=80 $ instead produces a qualitatively different trajectory and higher $\Delta v$, which is consistent with convergence to a different local minimum. This highlights how MPC performance is not monotonically guaranteed with horizon length for this non-convex problem. Hence, we use the tuned values of $N_{mpc} = 70$ horizon steps per solve, with $T_p=1750$ and $T_s=25$ seconds. The results of the MPC algorithm compared to the standard and adaptive Birkhoff methods is shown in \autoref{tab:comp}, and the resulting separation history is then shown in \autoref{fig:mpc_sep} where smaller separation violations are desired.

\begin{table}[h!]
\centering
\caption{Performance comparison of the standard (universal) LGL, adaptive LGL, and MPC algorithms}
\label{tab:comp}

\begin{tabular}{lccc}
\toprule
\rowcolor{white} 
  & \makecell{\textbf{Universal Birkhoff} \\ \textbf{(Phase 2a)}} & \makecell{\textbf{Adaptive Birkhoff} \\ \textbf{(Phase 2b)}} & \textbf{MPC} \\
\midrule
Total $\Delta v$ (km/s)        & 1.793 & 5.950 & 10.938 \\
\rowcolor{gray!15}
Total Wall Time (s)            & 6.8   & 18.8   & 258.7 \\
Separation Violation (km)        & 49.80   & 0.70   & 1.07 \\
\bottomrule
\end{tabular}
\end{table}

\begin{figure}[h!]
    \centering
    \includegraphics[width=\linewidth]{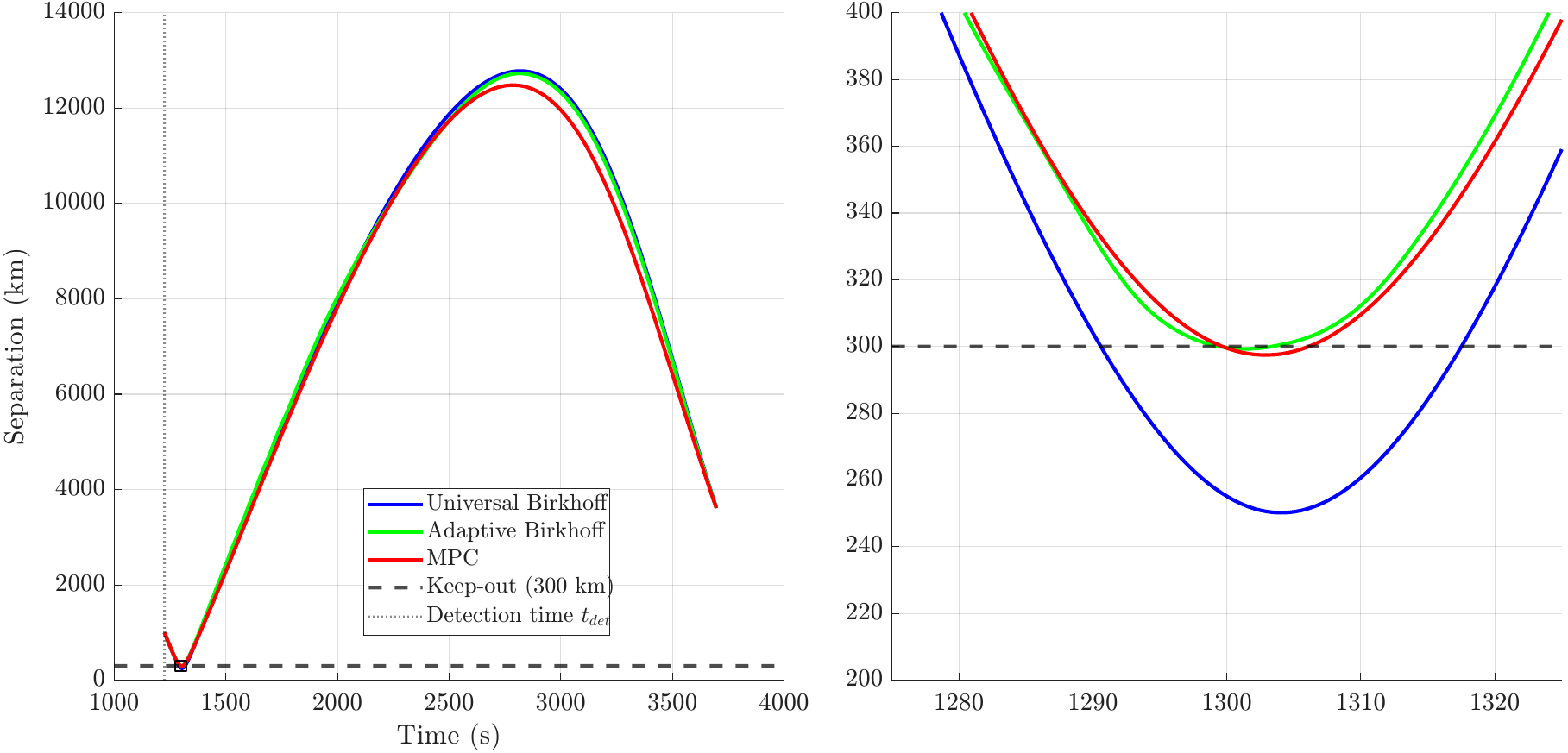}
    \caption{Left: The separation history of the MPC algorithm, compared to the Birkhoff histories. Right: zoomed-in version near the keep-out region.}
    \label{fig:mpc_sep}
\end{figure}

As \autoref{fig:mpc_sep} shows, the MPC algorithm achieves a qualitatively similar avoidance manoeuvre to the standard and adaptive Birkhoff methods. Initially, a high $\Delta v$ appears to make the MPC solution suboptimal, but like before, we argue that separation violation is the more critical metric compared to $\Delta v$ and wall time, since a violation of mission constraints carries greater consequence than added fuel expenditure or computation time. The MPC algorithm exhibits a slight keep-out constraint violation of $1.07$ km, which is comparable in scale to the adaptive method residual of 0.70 km, and substantially smaller than the universal Birkhoff method residual of 49.80 km established in the previous section. This small residual is a direct consequence of the soft-penalty formulation previously discussed, however, this trade-off is accepted in exchange for reliable solver convergence across the replanning horizon. 

Overall, using $N_{mpc}=70$ as the representative MPC configuration, it is clear that the adaptive Birkhoff method can achieve similar trajectories (in the locally-optimal sense) to MPC algorithms, while also offering improved computational efficiency (a reduction of 92.7\% in wall time from 258.7 sec by MPC to 18.8 sec by the adaptive Birkhoff method). Regarding fuel consumption, because MPC uses a prediction horizon, its visibility into the future is restricted, and its control authority over $t_\text{replan}$ accordingly yields a less globally efficient solution than the adaptive node formulation, which has full-horizon knowledge by construction. Therefore, the $\Delta v$ gap is attributable primarily to finite-horizon effects (ie. horizon myopia), rather than any limitation of the MPC implementation.

\section{Conclusion}
This paper presented an adaptive node distribution strategy for Birkhoff pseudospectral trajectory optimisation, in which collocation nodes are placed according to a probability density function derived from the thrust magnitude profile of the nominal solution. This was applied to a debris collision avoidance case study, where it was found that the proposed method was shown to produce a genuinely feasible avoidance trajectory, reducing the inter-node keep-out violation from $16.6\%$ from the standard LGL collocation to a comparatively negligible $0.23\%$, demonstrating the benefit of concentrating collocation resolution in regions of significant trajectory activity rather than relying only on uniform refinement of the LGL node set.

The results also revealed an important limitation of LGL node refinement, where increasing the number of collocation nodes did not consistently improve continuous-time feasibility. Instead, beyond a critical node count the changes in node placement from a higher $N$ could lead to irregular NLP solver behaviour. The adaptive formulation is also not immune to this behaviour with higher node counts also displaying irregular behaviour, but the principle advantage of the adaptive framework is in achieving substantially greater resolution of localised constraint activity with fewer nodes, which in turn reduces continuous-time constraint violations before encountering the irregular behaviour from these numerical difficulties.

Finally, when compared to a Model Predictive Control baseline, the adaptive method was able to achieve similar trajectories with a $92.7\%$ reduction in wall time, confirming that full-horizon pseudospectral optimisation outperforms reactive planning for this class of problem. These results demonstrate the potential of full-horizon pseudospectral optimisation to provide more efficient trajectory solutions than reactive planning for this class of debris collision avoidance problem, while also highlighting the importance of node placement and discretisation strategy when enforcing localised constraints.

The proposed framework is not limited to collision avoidance, and future work may involve applying this adaptive node method to a variety of other contexts. Because the adaptive node transformation is independent of the problem dynamics, it is directly applicable to any trajectory optimisation scenario where dynamic activity is localised in time, such as atmospheric entry and hypersonic re-entry, powered descent and landing, and low-thrust orbit raising. In each case, the key advantage of the Birkhoff formulation is that arbitrary node distributions are supported without sacrificing spectral accuracy or costate recovery, making the adaptive method a natural foundation for these extensions. In this way, the proposed method offers a computationally efficient alternative to iterative mesh refinement across a broad class of trajectory optimisation problems.

\appendix

\subsection*{Acknowledgements}
\label{acknowledgements}

This research was supported by the Commonwealth through an Australian Government Research Training Program Scholarship [DOI: \url{https://doi.org/10.82133/C42F-K220}].

\subsection*{Declaration of competing interest}

The authors have no competing interests to declare that are relevant to
the content of this article.

\section*{References}

\bibliographystyle{astrobib}
\bibliography{refs}

\end{document}

%% file: glossary.tex
\newacronym{lgl}{LGL}{Legendre-Gauss-Lobatto}
\newacronym{pdf}{PDF}{probability density function}
\newacronym{nlp}{NLP}{nonlinear programming}
\newacronym{cdf}{CDF}{cumulative distribution function}
\newacronym{psim}{PSIM}{pseudospectral integration matrix}
\newacronym{psdm}{PSDM}{pseudospectral differentiation matrix}
\newacronym{cmp}{CMP}{covector mapping principle}
\newacronym{leo}{LEO}{low-Earth orbit}
\newacronym{mpc}{MPC}{Model Predictive Control}
\newacronym{ipopt}{IPOPT}{Interior Point OPTimizer}